\documentclass[12pt,a4paper]{article}
\usepackage[dvipsnames]{xcolor}

\usepackage[british]{babel}

\usepackage[a4paper,top=2cm,bottom=2cm,left=2.5cm,right=2.5cm,marginparwidth=1.75cm]{geometry}

\usepackage[latin9]{inputenc}
\usepackage{amsmath}
\usepackage{graphicx}
\usepackage[colorlinks=true, allcolors=blue]{hyperref}
\usepackage{hyperref}
\usepackage[numbers, square]{natbib}
\usepackage[title]{appendix}
\usepackage{mathrsfs}
\usepackage{amsfonts}
\usepackage{booktabs} 
\usepackage{caption}  
\usepackage{threeparttable} 
\usepackage{algorithm}
\usepackage{algorithmicx}
\usepackage{algpseudocode}
\usepackage{listings}
\usepackage{enumitem}
\usepackage{chngcntr}
\usepackage{booktabs}
\usepackage{lipsum}
\usepackage{subcaption}
\usepackage{authblk}
\usepackage[T1]{fontenc}    
\usepackage{csquotes}       
\usepackage{diagbox}

\usepackage[T1]{fontenc}

\usepackage{amsmath}
\usepackage{amssymb}
\usepackage{babel}
\usepackage{amsthm}
\usepackage{graphicx}
\usepackage{tikz}
\usepackage[all]{xy}
\usepackage{bm}
\usepackage{comment}
\usetikzlibrary{calc,patterns,angles,quotes}
\usepackage[percent]{overpic}

\usepackage{setspace}
\usepackage{titlesec}
\titleformat{\section} 
  {\normalfont\Large\bfseries}{\thesection.}{1em}{}
  
\newcommand{\R}{\mathbb{R}}

\newcommand{\Z}{\mathbb{Z}}

\newcommand{\N}{\mathbb{N}}

\newtheorem{theorem}{Theorem}[section]
\newtheorem{prop}[theorem]{Proposition}
\newtheorem{lemma}[theorem]{Lemma}

\newtheorem{rem}[theorem]{Remark}
\newtheorem{cor}[theorem]{Corollary}
\newtheorem{defi}[theorem]{Definition}

\usepackage{float}   
\usepackage{caption} 


\title{Chaotic Boltzmann's Billiard Systems at positive energy}

\author[1]{Irene De Blasi}
\author[2]{Airi Takeuchi}
\author[1]{Susanna Terracini}
\affil[1]{\small Department of Mathematics, University of Turin, Turin, Italy}
\affil[2]{Institute of Mathematics, University of Augsburg, Augsburg, Germany}

\date{}  
\begin{document}
\maketitle

\begin{abstract}
This paper deals with the so-called Boltzmann billiard, that is, a billiard subjected to a central force of the type $V(r)=-\alpha/r-\beta/r^2$, $\alpha$ and $\beta$ being real constants, and with a straight reflection table. In the particular case of $\alpha$ and $\beta$ positive, we prove the presence of a symbolic dynamics, and hence of positive topological entropy, at positive energy for $\beta$ sufficiently small. 
\end{abstract}



\section{Introduction}

In \cite{Boltzmann}, Boltzmann introduced a billiard model in $\R^2$ governed by a Keplerian potential with an inverse-square correction. The model features a reflecting wall at a fixed distance from the origin and combines gravitational and inverse-square forces, providing a framework for his ergodic hypothesis. 
More precisely, the dynamics are generated by the centrally symmetric potential
\begin{equation}\label{eq:Valphabeta}
V:\R^2\setminus\{0\}\to \R, \quad V(z)=-\dfrac{\alpha}{|z|}-\dfrac{\beta}{|z|^2}, 
\end{equation}
with $\alpha,\beta\in \R$, and reflections at a straight wall (see Figure \ref{fig:intro}).  
\begin{figure}       
\includegraphics[width=0.7\textwidth]{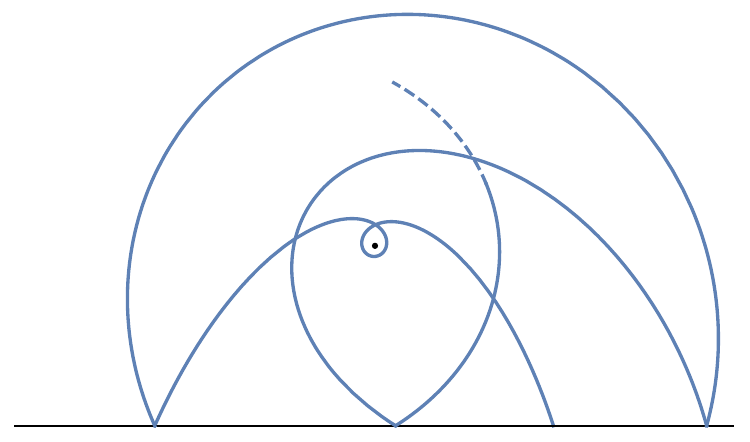}
\caption{Boltzmann billiard with $L=1, \alpha=1$, and $\beta=0.3$.}
\label{fig:intro}
\end{figure}
Without reflections, the system is integrable, with first integrals given by the energy $h$ and angular momentum $C$. Boltzmann suggested that elastic reflections destroy conservation of $C$ and hence integrability, leading the particle to explore the whole accessible region uniformly. This was meant as a concrete illustration of the \emph{ergodic hypothesis}, later reformulated by Ehrenfest-Afanassjewa \cite{Ehrenfest1911} as the existence of trajectories passing arbitrarily close to every point of the phase space (see \cite{Reichert2024, Birkhoff1931}).  

 More than a century later, Gallavotti and Jauslin \cite{Gallavotti-Jauslin} showed that Boltzmann's claim was actually incorrect: they proved integrability in the Keplerian case (\(\beta=0\)) by constructing an additional first integral involving $C$ and the components of Laplace-Runge-Lenz vector. Related results were obtained by Felder, Zhao, and others \cite{Zhao2021,Takeuchi-Zhao2023,Gasiorek}, who further analyzed the integrable structure and periodic orbits, as well as the possibility to apply KAM theory (\cite{kolmogorov1954hamiltonian, arnold1963kolmogorov, moser1962invariant}) to the case $|\beta|\ll1$.  

Our work is focused on the non-Keplerian case $\beta>0$. Our aim is to prove that for small positive $\beta$, Boltzmann's billiard exhibits chaotic dynamics if the particle moves on the same side of the origin. To this end we employ symbolic dynamics, which encodes trajectories into sequences of integers representing windings around the center. This provides a discrete model for describing chaos in continuous systems.  

The study of chaotic behavior in systems with central force potentials has long been central to dynamical systems. Foundational works include Poincar\'e's analysis of the three-body problem \cite{poincare1891trois}, the KAM theory, and the H\'enon-Heiles model \cite{Henon1964}, all of which showed how small perturbations may drastically change trajectory structures and generate chaos. Such effects are particularly relevant in celestial mechanics, where orbital instability and long-term evolution are key phenomena. In this context, the potential with an inverse-square correction arises naturally in the relativistic Kepler problem, where relativistic effects modify the classical potential \cite{boscaggin2023maupertuis}.  

Determining integrability of billiard systems remains central in dynamics, from the classical Birkhoff case and the \emph{Birkhoff-Poritsky conjecture} \cite{Birkhoff1927,Poritsky1950,KaloshinSorrentino2018} to billiards with non-constant or singular potentials \cite{Panov1994EllipticalBilliard,Dragovic1998,fedorov2001ellipsoidal,Bolotin2017DegenerateBilliards,JAUD2024105289,TAKEUCHI2024109411,BaranziniBarutelloDeBlasiTerracini2025}. For integrable billiards, the discrete Arnold-Liouville theorem \cite{veselov1991integrable} ensures the phase space is foliated by invariant tori, while non-integrable ones may display chaos.  Recently, in the works \cite{barutello2023chaotic, BaranziniBarutelloDeBlasiTerracini2025} a  billiard model with Keplerian potential and bounded reflecting wall was studied, proving the existence of topological chaos at high energies by explicitly constructing the associated symbolic dynamics. Inspired by their approach, we use variational methods to establish chaotic dynamics for Boltzmann's billiard for $h>0$ and $\beta>0$ small.  

Considering the potential $V$ in \eqref{eq:Valphabeta}, we define the reflecting wall as $\ell := \{y = -L\} \subset \R^2$,  $L > 0$ denoting the distance from the origin. We further assume that the particle's total energy $h$ is positive. After a suitable rescaling in space and time, our system is equivalent to the one with potential
\[
\tilde V(z)= -\frac{\tilde\alpha}{|z|}- \frac{ \tilde\beta}{ |z|^2}, \quad \tilde \alpha = \frac{\alpha}{L h}, \quad \tilde \beta = \frac{\beta}{L^2h}= \tilde \alpha \frac{\beta}{L \alpha}, 
\]
with normalized energy $\tilde h=1$ and distance $\tilde L=1$. To emphasize the role of physical parameters, however, we keep $h$ and $L$ explicit and set $\tilde \alpha=1$. Thus, we focus on the billiard system generated by  
\begin{equation}\label{eq:pot}
    V_{\beta}: \R^2\setminus\{0\}\to \R, \quad V_{\beta}(z)=-\dfrac{1}{|z|}-\dfrac{\beta}{|z|^2}, \quad \beta>0.
\end{equation}

We show that if $\beta$ is sufficiently small the system admits a symbolic dynamics, defined, according to \cite{hasselblatt2003first} as follows.  

\begin{defi}
Let $f: A \to A $ be an area-preserving map on a set $A$. We say that $f$ admits a symbolic dynamics over $A$ if there exists a (finite or countable) set of symbols $\mathcal S$, a set of bi-infinite words $\Omega=\mathcal S^\mathbb Z$, and a projection $\pi: A \to \Omega$ such that 
    \begin{enumerate}
        \item  $\pi: A \to \Omega$ is continuous and surjective,  
        \item $\sigma \circ \pi= \pi\circ f$ on $A'$, 
    \end{enumerate}
where $\sigma: \Omega\to\Omega$ is the Bernoulli shift, defined by $\sigma((s_{i})_i) = (s_{i+1})_i$.  
In such case, we say that $f$ is semi-conjugated to $\sigma$. 
\end{defi}

We stress that in our setting, with the particle moving in the same side as the centre, addinga strong-force term with coefficient $\beta$ small \emph{is not} a perturbation of the pure Keplerian case. Indeed, the distance from the origin is not bounded away from zero, so that close to the centre the second term in $V_\beta$ dominates and the dynamics is far from being a perturbative regime of $\beta=0$.  

The symbolic dynamics we construct associates billiard trajectories with bi-infinite sequences of integers, encoding the winding number of each arc with respect to the origin.   

\begin{theorem}\label{thm:final}
Let $F: X\to X$ be the billiard map associated with the Boltzmann system, where $X$ denotes the set of initial conditions $(p, v)\in \ell\times \R^2$, $v\cdot(0, 1)>0$, for which $F$ is well defined. Then, if $\beta>0$ is sufficiently small, there exists a subset $X' \subset X$, invariant under $F$, such that $F_{|_{X'}}$ is semi-conjugated to the Bernoulli shift on $\Omega=\left(\Z\setminus\{0\}\right)^\Z$.  \\
In other words, the Boltzmann billiard admits a symbolic dynamics with countably many symbols on a suitable subsystem. 
\end{theorem}
\noindent As a consequence, the system has \emph{infinite} topological entropy (see e.g.~\cite{brin2002introduction}).  
\begin{cor}\label{cor:final}
   For sufficiently small $\beta>0$, the Boltzmann billiard has infinite topological entropy. 
\end{cor}
\noindent We point out that, although we work on the normalised system with potential $\tilde V$, Theorem \ref{thm:final} and Corollary \ref{cor:final} work also in the general case $\alpha, h, L>0$. The mutual interaction between such parameters will be clarified in Section \ref{sec:conclusions}. 

The proof of Theorem \ref{thm:final} proceeds in two steps. First (Section \ref{sec:fix_ends}), we establish the existence of fixed-end solutions of the mechanical system associated with $V_\beta$, connecting two points of $\ell$ while satisfying prescribed topological constraints. This is done via a variational approach, relating such solutions to critical points of Maupertuis' functional \cite{ambrosetti2012periodic} in a suitable function space. A key issue is to show that arcs with endpoints on $\ell$ do not intersect $\ell$ elsewhere. We address this using a perturbative argument, starting from geometric properties of hyperbolic Keplerian arcs and extending them to the case $\beta>0$ and small. In the second step (Section \ref{sec:symDyn}), we build the symbolic dynamics on top of these variational solutions. Section \ref{sec:conclusions} will connect our results to the existing literature, underlying open problems and possible generalizations. Finally, Appendix \ref{sec:appA} will present a more general result, relating minimisers of Maupertuis functional and Jacobi length, used in Section \ref{sec:fix_ends} to exclude collisional arcs.

\section{Fixed-ends trajectories}\label{sec:fix_ends}

The first step in constructing a symbolic dynamics consists in finding suitable solutions of the fixed-ends problem associated with $V_\beta$, for positive energies and prescribed topological properties. 
We define the reflection wall as the line $\ell=\{(x,y)\in\R^2 \mid y=-L\}$ with $L>0$, and denote by 
\[
\mathcal P=\{(x,y)\in \R^2 \mid y>-L\}
\] 
the corresponding half-plane containing the origin.  
We are interested in solutions of the mechanical system associated with $V_\beta$ which start and end on $\ell$, while remaining entirely in $\mathcal P$. More precisely, given $z_0, z_1\in \ell$, we search for solutions of the problem 
\begin{equation}\label{eq:fixed_ends_prob}
    \begin{cases}
        \ddot z(t)=-\nabla V_\beta\!\left(z(t)\right)=\dfrac{z}{|z|^3}+2\beta\dfrac{z}{|z|^4}, & t\in [0, T],\\[6pt]
        \dfrac{|\dot z(t)|^2}{2}+V_\beta(z(t))=h, & t\in [0, T],\\[6pt]
        z(0)=z_0, \quad z(T)=z_1, & \\[6pt]
        z(t)\in\mathcal P, & t\in(0, T),
    \end{cases}
\end{equation}
for some $T>0$.  

To construct a symbolic dynamics, we classify the solutions of \eqref{eq:fixed_ends_prob} according to their \emph{winding number}, defined as follows. 

\begin{defi}\label{def:wind}
   Let $u:[a,b]\to \overline{\mathcal P}$ be such that $u(a), u(b)\in \ell$ and $u(s)\in \mathcal P\setminus\{0\}$ for every $s\in (a, b)$. Consider the closed curve $\gamma_u$ obtained by concatenating $u$ with the oriented segment from $u(b)$ to $u(a)$ (see Figure \ref{fig:wind}). The \emph{winding number} of $u$ with respect to the origin is defined as 
   \[
   \operatorname{Ind}(u) = \frac{1}{2 \pi i} \oint_{\gamma_{u}} \frac{du}{u} \in \Z.
   \] 
\end{defi}

This definition applies directly to non-collisional solutions of \eqref{eq:fixed_ends_prob}, providing a classification criterion for constructing the symbolic dynamics.  

\begin{figure}
        \centering        
        \begin{overpic}[width=0.7\textwidth]{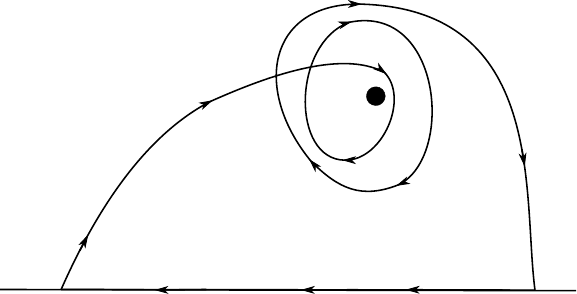}
        \put (102,0) {\rotatebox{0}{$\ell$}}
        \put (3,2) {\rotatebox{0}{$u(a)$}}
        \put (93,2) {\rotatebox{0}{$u(b)$}}
        \put (23,27) {\rotatebox{50}{$u(s)$}}
        \put (60,2) {\rotatebox{0}{$\gamma_u$}}
        \end{overpic}
\caption{Closed curve $\gamma_u$ constructed to compute the winding number of $u$ in Definition \ref{def:wind}. }
        \label{fig:wind}
\end{figure}
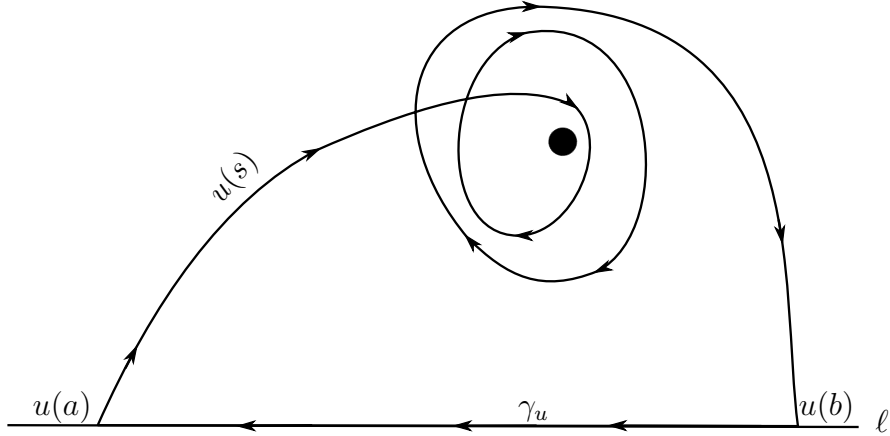

To find solutions of \eqref{eq:fixed_ends_prob} with prescribed winding numbers, we apply the direct method in the calculus of variations, searching for minimisers of the \emph{Maupertuis's functional} within a suitable class of paths. For this, compactness arguments are required, so we restrict the motion to a compact subset $D\subset\mathcal P$, bounded by $\ell$ and by a particular solution of \eqref{eq:fixed_ends_prob}.  In the following Lemma, and later on, whenever we do not specify otherwise the parameters $h, L, \beta$ are aribitrary and positive.

\begin{lemma}
    \label{lem: boundary_existence}
    There exists a solution $\hat{\mathcal{Z}}: \R \to \R^2$ of the mechanical system associated with $V_\beta$ at energy $h$ such that: 
    \begin{itemize}
    \item $\hat{\mathcal{Z}}(0), \hat{\mathcal{Z}}(T)\in \ell$ for some $T>0$, and $\hat{\mathcal{Z}}(t)\in \mathcal P$ for all $t\in(0, T)$;  
    \item $\angle \!\left(\dot{\hat{\mathcal{Z}}}(0), (1,0)\right)<\pi/2$ and $\angle\!\left((-1, 0),\dot{\hat{\mathcal{Z}}}(T)\right)>\pi/2$ (see Figure \ref{fig:hat_z}); 
    \item $\operatorname{Ind}(\hat{\mathcal Z})=1$.
    \end{itemize}
\end{lemma}

\begin{proof}
At large distances from the center, the potential $V_\beta$ may be regarded as a small perturbation of the Keplerian potential $V_{\text{kep}}(z)=-1/|z|$. For $V_{\text{kep}}$, the statement holds, since one can construct infinitely many hyperbolic Keplerian arcs at energy $h$ satisfying the required properties, with pericentres arbitrarily far from $0$. By differentiable dependence of ODEs on parameters, and by the regularity of both $V_{\text{kep}}$ and $V_\beta$ away from the origin, the claim follows.
\end{proof}

\begin{figure}
        \centering        
        \begin{overpic}[width=0.7\textwidth]{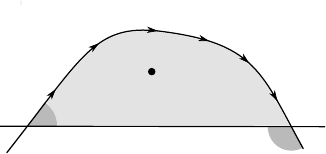}
        \put (-3,8) {\rotatebox{0}{$\ell$}}
        \put (8,4) {\rotatebox{0}{$(\hat x_0, -L)$}}
        \put (89,10) {\rotatebox{0}{$(\hat x_1, -L)$}}
        \put (17,27) {\rotatebox{60}{$\hat{\mathcal Z}(s)$}}
        \put (89,30) {\rotatebox{0}{$\mathcal P$}}
        \put (65,20) {\rotatebox{0}{$D$}}
        \end{overpic}
\caption{Bounding trajectory $\hat{\mathcal Z}$ from Lemma \ref{lem: boundary_existence}, along with the domain $D\subset\mathcal P$.}
        \label{fig:hat_z}
\end{figure}
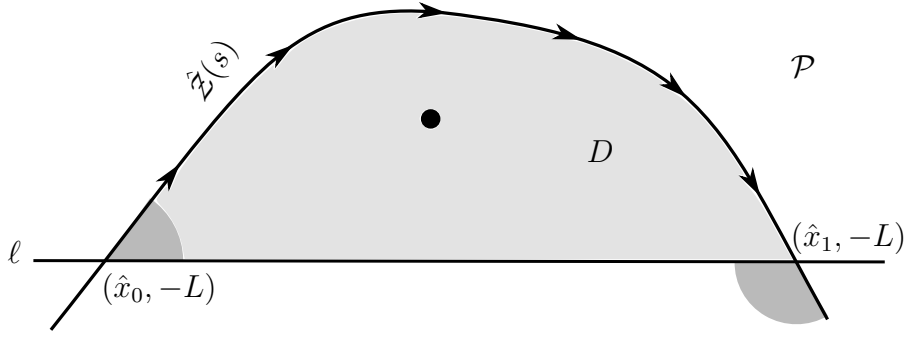

The trajectory $\hat{\mathcal{Z}}$ needs not to be  unique; nevertheless, it suffices to fix one such trajectory for the direct method argument.  

\medskip

Let us now select an orbit $\hat{\mathcal{Z}}$ intersecting $\ell$ at two points $\hat{\mathcal{Z}}(t_0)=(\hat{x}_0, -L)$ and $\hat{\mathcal{Z}}(t_1) = (\hat{x}_1, -L)$. Let $\ell' \subset \ell$ be the line segment with endpoints $\hat{\mathcal{Z}}(t_0)$ and $\hat{\mathcal{Z}}(t_1)$, and denote by $D$ the compact region bounded by $\hat{\mathcal{Z}}$ and $\ell$ (see Figure \ref{fig:hat_z}). Observe that, fixing $L$ and $h$ if $\beta>0$ is bounded from above by a suitable constant, it is possible to find a compact segment $\tilde \ell\subset\ell$ such that, for all $\beta$ under consideration, one has $\ell'\subset \tilde\ell$.  

We now restrict attention to generic non-collisional paths starting and ending on $\ell'$. The notion of winding number provides a topological classification.  

\begin{defi}
    Given $z_0, z_1\in\ell'$ and $k\in \mathbb Z\setminus \{0\}$, define
\[
\hat{\mathcal H}^k_{z_0 z_1} = \{u \in H^1([0,1],D) \mid u(0) = z_0, u(1) = z_1, \ u(s) \neq 0 \ \forall s \in [0,1],\  \operatorname{Ind}(u) = k \}.
\]
The set $\hat{\mathcal H}^k_{z_0 z_1}$ is open in $\mathcal C^0$. One can also consider the class of collisional paths
\[
Coll_{z_0 z_1}= \{u \in H^1([0,1],D) \mid u(0) = z_0, u(1) = z_1,\ \exists s \in [0,1]: u(s) = 0 \}.
\]
It is easy to check that $Coll_{z_0z_1}=\partial \hat{\mathcal H}_{z_0z_1}$ for each $k\neq 0$, so we define the closure
\[
\mathcal H^k_{z_0 z_1} = \hat{\mathcal H}^k_{z_0 z_1} \sqcup Coll_{z_0 z_1}.
\]
\end{defi}

\begin{rem}\label{rem:H_weak_closed}
    Since weak convergence in $H^1$ implies uniform convergence, the set $\mathcal H_{z_0z_1}^k$ is weakly closed in the $H^1$ topology. 
\end{rem}

We are now ready to introduce the functional we aim to minimise by the direct method. We use the so-called Maupertuis functional (see \cite{ambrosetti2012periodic}), in view of the correspondence between its critical points and solutions of \eqref{eq:fixed_ends_prob}.     

\begin{defi}
    The \emph{Maupertuis' functional} $\mathcal{M}: \mathcal H^k_{z_0 z_1} \to \R \cup \{\infty\}$ is defined as
    \[
    \mathcal{M}(u) =  \frac{1}{2} \int_{0}^{1} |\dot{u}(s)|^2 ds \int_{0}^{1} (h - V_{\beta}(u(s))) ds.
    \]
\end{defi}

\begin{theorem}[\cite{ambrosetti2012periodic}]\label{thm:M_crit_point}
    If $u: [a,b]\to \R^2$ is a collision-free critical point of $\mathcal M$ in 
\[
\{v\in H^1([a,b],\R^2) \mid v(a)=z_0, \ v(b)=z_1\}, 
\]
such that $\mathcal M(u)>0$, then, setting 
\[
\omega^2=\dfrac{\int_a^b(h-V_\beta(u(s)))\,ds}{\int_a^b|\dot u(s)|^2\,ds},
\]
the reparametrisation $z(t)=u(\omega t)$ is a classical solution of 
\begin{equation}\label{eq:differential}
    \begin{cases}
\ddot{z}(t)=-\nabla V_{\beta}(z(t)), & t\in [0,T],\\[6pt]
\dfrac{|\dot{z}(t)|^2}{2}+V_\beta(z(t))=h, & t\in[0,T],\\[6pt]
z(0)=z_0,\quad z(T)=z_1,
\end{cases}
\end{equation}
for some $T>0$.
\end{theorem}

We now establish some properties of $\mathcal M$ on $\mathcal H^k_{z_0z_1}$, $k\in\Z\setminus\{0\}$.  

\begin{prop}\label{prop:M_prop}
    For every $z_0,z_1\in\ell'$ and $k\in\Z\setminus\{0\}$: 
    \begin{enumerate}
        \item there exists $C=C(k)>0$ such that $\mathcal M(u)\geq C$ for all $u\in \mathcal H^k_{z_0z_1}$; 
        \item $\mathcal M$ is weakly lower semicontinuous (w.l.s.c.) on $\mathcal H^k_{z_0z_1}$; 
        \item $\mathcal M$ is coercive on $\mathcal H^k_{z_0z_1}$. 
    \end{enumerate}
\end{prop}

\begin{proof}
(1) Since 
    \[
     \int_{0}^{1} (h - V_{\beta}(u(s)))\,ds = \int_{0}^{1} \left(h + \frac{1}{|u(s)|} + \frac{\beta}{|u(s)|^2}\right) ds > h, 
    \]
    it suffices to show that there exists $C>0$ such that 
    $\| \dot{u}\|^2_{2} \geq C,$
    where $\|\cdot\|_2$ is the standard $L^2$ norm.  
    Assume by contradiction that there exists a sequence $\{u_n\}$ in $\mathcal H^k_{z_0, z_1}$ with $\|\dot{u}_n\|_2 \to 0$. Since $\{u_n\}$ is bounded, up to subsequences it converges uniformly to some $v \in \mathcal H^k_{z_0 z_1}$.  
    Write $u_n (s) = \bar{u}_n + w_n (s),$ with $\bar{u}_n = \int_{0}^{1} u_n(s)\,ds$. Then $\|\dot w_n\|_2 \to 0$.  
    By the Poincar\'e--Wirtinger inequality,  
    $
    \|w_n \|_2 \leq \|\dot{w_n} \|_2,
    $
    so $\|w_n\|_2 \to 0$. This implies $v(s)$ is constant for almost every $s\in [0,1]$, contradicting the assumption $k\neq 0$.  

(2) It is known that $\mathcal{M}$ is w.l.s.c.~iff the sublevel sets 
    \[
    M^{c}:= \{ u \in \mathcal H^k_{z_0 z_1}  \mid \mathcal{M}(u) \leq c\}
    \]
    are weakly closed for all $c>0$ (see \cite{brezis2011functional}). 
    Fix $c>0$ and take a sequence $\{u_n\}\subset M^{c}$ converging to $u \in \mathcal H^k_{z_0 z_1}$ weakly in $H^1$. Since the $H^1$ norm is w.l.s.c.,  
    \[
    \|u\|^2_2 + \|\dot{u}\|^2_2 \leq \liminf_{n \to \infty} (\|u_n\|^2_2 + \|\dot{u}_n\|^2_2 ).
    \]
    By uniform convergence of $u_n\to u$, we also have
    \[
     \|\dot{u}\|^2_2 \leq \liminf_{n \to \infty} \|\dot{u}_n\|^2_2.
    \]
  Fatou's lemma then yields
    \[
    \int_{0}^{1} (h - V_{\beta}(u(s))) ds \leq \liminf_{n \to \infty} \int_{0}^{1} (h - V_{\beta}(u_n(s))) ds. 
    \]
    Combining these estimates,  
    \[
    \begin{split}
    \mathcal M(u) &= \tfrac{1}{2} \|\dot{u}\|_2^2 \int_{0}^{1} (h - V_{\beta}(u(s))) ds 
     \leq \tfrac{1}{2}  \liminf_{n \to \infty } \left( \|\dot{u_n}\|_2^2 \int_{0}^{1} (h - V_{\beta}(u_n(s))) ds \right) \\
     &\leq \liminf_{n \to \infty } \mathcal{M} (u_n) 
     \leq c.
    \end{split}
    \]

(3)  Take a sequence $\{u_n\}\subset \mathcal H^k_{z_0 z_1}$ with $\|u_n\|_{H^1} \to \infty$. Since $\|u_n\|_2$ is bounded, it follows that $\|\dot{u_n}\|_2^2 \to \infty $. Moreover,
    \[
    \int_{0}^{1} (h - V_{\beta}(u_n(s))) ds >0,
    \]
    so $\mathcal{M}(u_n) \to \infty$.
\end{proof}

Remark \ref{rem:H_weak_closed}, together with Proposition \ref{prop:M_prop}, allows us to apply the direct method in the calculus of variations, yielding the following existence result.

\begin{theorem}
\label{thm: exsistence_minimizer}
\textcolor{red}{Parametri}    
Take $\hat{\mathcal{Z}}$ as in Lemma \ref{lem: boundary_existence}. For any $z_0,z_1 \in \ell'$, there exists $u \in \mathcal H^k_{z_0 z_1}$ such that
    \[
    \mathcal{M}(u) = \min_{v \in \mathcal H^k_{z_0 z_1}} \mathcal{M}(v).
    \]
\end{theorem}

The minimiser of Theorem \ref{thm: exsistence_minimizer} needs not to be unique; however, all minimisers of $\mathcal M$ in $\mathcal H_{z_0z_1}^k$ share certain properties. For instance, they are collision-free. To prove this, we introduce the notion of Jacobi length (see \cite{ambrosetti2012periodic}).  

\begin{defi}
    For a collision-free curve $u \in \hat{\mathcal H}^k_{z_0 z_1}$ we define the \emph{Jacobi length} as
    \[
    \mathcal{L}(u) = \int_{0}^{1} |\dot{u}(s)|\sqrt{h - V_{\beta}(u(s))}\, ds.
    \]
\end{defi} 

Jacobi length and the Maupertuis' functional are closely related. By the Cauchy--Schwarz inequality,  
\begin{equation}\label{eq:L_M_ineq}
    \mathcal{L}(u) \leq \sqrt{\int_{0}^{1} |\dot{u}(s)|^2 ds}\cdot \sqrt{\int_{0}^{1}(h- V_{\beta}(u(s))) ds} =\sqrt{2\mathcal{M}(u)},
\end{equation}
with equality iff there exists $m\in \R$ such that
\begin{equation}
        |\dot u(s)|^2 = m \big(h-V_\beta(u(s))\big) 
\end{equation}
for almost every $s\in[0,1]$.  

\noindent
For a collision-free curve, $\mathcal L(u)$ can be interpreted as its length in the Jacobi metric induced by $V_\beta$; in particular, it is related to the pericentre of the trajectory.  

\begin{lemma}
    \label{lem: Jacobi_lowerbound}
    For every $u \in \hat{\mathcal H}_{z_0 z_1}$,
    \begin{equation}\label{eq:pericentre}
    \min_{s \in [0,1]}| u(s)| >   L\exp{\!\left(-\frac{\mathcal{L}(u)}{\sqrt{\beta}}\right)}, 
    \end{equation}
    where $L=\mathrm{dist}(0, \ell)$. 
\end{lemma}

\begin{proof}[Proof of Lemma \ref{lem: Jacobi_lowerbound}]
Let $\tilde{s} \in [0,1]$ be such that $|u(\tilde{s})| = \min_{s \in [0,1]} |u(s)|$.  
Define $U:\R^2\setminus\{0\}\to \R$ by $U(z)=\sqrt{\beta}\log(|z|)$. Then 
\[
- V_{\beta}(z) > \frac{\beta}{|z|^2} = | \nabla U(z)|^2, \quad z\neq 0.
\]
Hence,
\begin{align*}
    \mathcal{L}(u) &> \int_{\tilde{s}}^{1}  \left|\frac{d}{ds} U (u(s)) \right| ds 
    = \sqrt{\beta} \log \frac{|u(1)|}{|u(\tilde{s})|}. 
\end{align*}
Since $u(1)=z_1\in \ell$, we obtain
\[
\mathcal L(u)>\sqrt{\beta}\log\!\left(\frac{L}{|u(\tilde s)|}\right),
\]
which is equivalent to \eqref{eq:pericentre}.
\end{proof}

\begin{theorem}\label{thm:No_Coll}
    Let $u$ be a minimiser of $\mathcal{M}$ in $\mathcal H^k_{z_0z_1}$. Then $u$ is collision-free. 
\end{theorem}

\begin{proof}
Since $\mathcal M(u)<\infty$, Lemma \ref{lem: Jacobi_lowerbound} and inequality \eqref{eq:L_M_ineq} imply that for every $v\in \hat{\mathcal H}_{z_0z_1}$,
\[
\min_{s\in[0,1]} |v(s)|>L \exp\!\left(-\sqrt{\frac{2\mathcal M(v)}{\beta}}\right). 
\]
Take a minimizing sequence $\{u_n\}\subset\hat{\mathcal H}^k_{z_0z_1}$. Then $\mathcal{M}(u_n)\to \mathcal{M}(u)$ and $u_n\rightharpoonup u$ in $H^1$, hence $u_n\to u$ uniformly. Consequently, 
\[
\min_{s\in[0,1]} |u(s)|=\lim_{n\to\infty}\min_{s\in[0,1]}|u_n(s)|\geq \lim_{n\to\infty}\mathcal M(u_n)=\mathcal M(u)>0, 
\]
so $u$ is collision-free.
\end{proof}

Since any minimiser $u$ of $\mathcal{M}$ belongs to $\hat{\mathcal H}^k_{z_0 z_1}$, the Jacobi length $\mathcal{L}(u)$ is well-defined. Following the reasoning in Appendix A, Proposition \ref{prop:minML}, one can further show that minimisers of $\mathcal M$ also minimise $\mathcal L$.  

\begin{cor}\label{cor:minMminL}
    The minimiser of Theorem \ref{thm: exsistence_minimizer} also minimises $\mathcal L$ on the same set.  
\end{cor}

\medskip

We now pass from minimisers of $\mathcal M$ to solutions of \eqref{eq:differential}, with prescribed endpoints in a compact subset of $\ell$ and winding number $k\neq 0$. In view of Theorems \ref{thm:M_crit_point}, \ref{thm: exsistence_minimizer} and \ref{thm:No_Coll}, it suffices to prove that a minimiser of $\mathcal M$ in $\mathcal H_{z_0z_1}^k$ does not intersect $\partial D$, except at its endpoints. This requires assuming $\beta$ sufficiently small.  

\begin{prop}
    There exists $\bar\beta=\bar\beta(L, h, \alpha)>0$ such that for every $\beta<\bar\beta$, $z_0, z_1\in \ell'$, and $k\in\mathbb Z\setminus\{0\}$, any minimiser of $\mathcal L$ in $\mathcal H^k_{z_0z_1}$ does not touch $\partial D$. 
\end{prop}

\begin{proof}
Fix $z_0, z_1\in \ell'$, $k\neq 0$, and $\beta>0$, and let $u(s)$ be a minimiser of $\mathcal L$ in $\mathcal H^k_{z_0z_1}$. We already know that $u$ is collision-free. By the classical theory of geodesics with obstacles (see \cite{marino1983geodesics}), $u$ could touch $\partial D$ only tangentially.  

Tangency with the orbit $\hat{\mathcal{Z}}$ is impossible, since $\hat{\mathcal Z}$ is itself a solution of the system with potential $V_\beta$ (Lemma \ref{lem: boundary_existence}), and uniqueness of solutions of ODEs with prescribed initial conditions rules out such occurence.  

We are left to exclude tangency with $\ell'$. To this end, we show that if $\beta$ is sufficiently small, then every solution of \eqref{eq:diffProb} starting on $\ell'$ with initial velocity tangent to $\ell$ does not intersect $\ell$ again (see Figure \ref{fig:initialOrb}).  

Consider the Cauchy problem
\begin{equation}\label{eq:diffProb}
    \begin{cases}
        \ddot z(t)=-\nabla V_\beta(z(t)),\\
        z(0)=z_0\in \ell',\\
        \dot z(0)=v_0 \parallel (1, 0).
    \end{cases}
\end{equation}
Let $\phi$ be the angle between the position vector $z_0$ and the positive direction along $\ell'$ (Figure \ref{fig:initialOrb}). The angular momentum of $z(t)$ is then
\begin{equation}
    C=|z_0||v_0|\sin\phi, 
\end{equation}
which can also be expressed as
\begin{equation}
    C^2=2h L^2+2L\sin\phi+2\beta\sin^2\phi. 
\end{equation}

\begin{figure}
        \centering        
        \begin{overpic}[height=0.2\textheight]{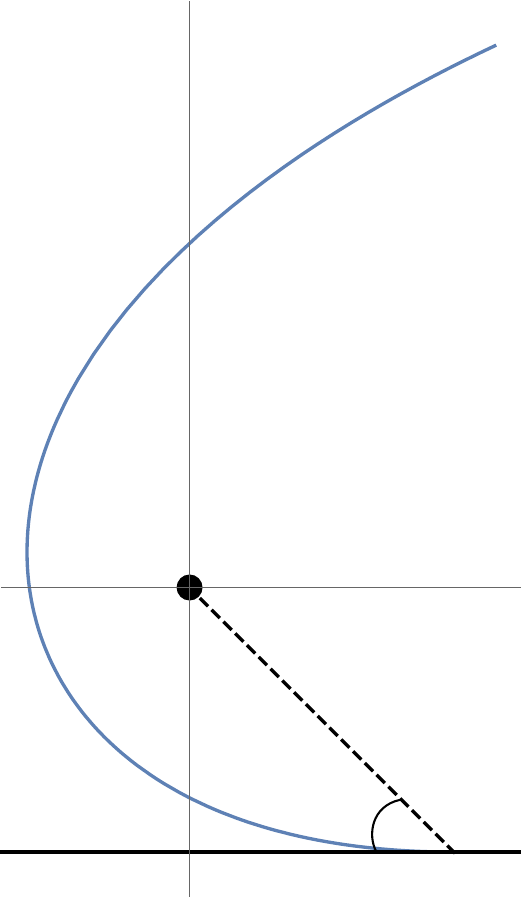}
		\put (33,10) {\rotatebox{0}{$\phi$}}
        \put (60,5) {\rotatebox{0}{$\ell'$}}
        \end{overpic}
\caption{Example of a solution tangent to the line $\ell'$.}
        \label{fig:initialOrb}
\end{figure}

Since $\phi\in(0,\pi)$, we obtain a uniform bound on $C$ depending on $\beta$:  
\[
    2hL^2<C^2<2hL^2+2L^2+2\beta. 
\]
Fix $h>0$ and set $\beta_1=L^2h$. For every $\beta>\beta_1$ and $z_0\in \ell$, the angular momentum of solutions of \eqref{eq:diffProb} tangent to $\ell$ lies in the compact set 
\[
	K := [-\sqrt{4L^2h + 2L}, - \sqrt{2L^2h}] \cup [\sqrt{2L^2h}, \sqrt{4L^2h + 2L}].
\]

The effective potential associated with $V_\beta$ (see \cite{jose1998classical}) is
\[
    V_{eff}(r)=-\frac{1}{r}+\frac{C^2-2\beta}{2r^2}, 
\]
and the radial equations are
\begin{equation}\label{eq:PolEq}
    \ddot{r}=-\dfrac{d V_{eff}(r)}{dr}, \quad \frac{1}{2}\dot r^2=h - V_{eff}(r). 
\end{equation}
Since $C^2>2\beta$, the trajectory has a pericentre at 
\[
r_{min}=\frac{-1+\sqrt{1+2(C^2-2\beta)h}}{2h}
\]
after which it escapes to infinity.  

Let $\hat\theta(h, \beta, C)$ denote the total polar angle swept from pericentre to infinity. From \eqref{eq:PolEq} and angular momentum conservation,
\begin{align}\label{eq:angleBeta}
	\hat{\theta}(h, \beta, C) 
    &= \int_{r_{min}}^{\infty} \frac{C}{\sqrt{2}r^2}\cdot \frac{dr}{\sqrt{\displaystyle h + \frac{1}{r} - \frac{C^2 - 2\beta}{2r^2}}} \\ 
    &=\frac{C}{C'} \cdot \hat\theta(h, 0, C'), \nonumber
\end{align}
where $C' = \sqrt{C^2 - 2\beta}$ and 
\[
\hat\theta(h, 0, C') = \frac{C'}{\sqrt{2}}\int_{r_{min}}^{\infty}\frac{dr}{r^2\sqrt{h+ \dfrac{1}{r} - \dfrac{C'^2}{2r^2}}}
\]
is the maximal central angle of a Keplerian orbit with energy $h$ and angular momentum $C'$.  

Now consider a Keplerian hyperbola with these properties. Let $\delta\in(0,\pi/2)$ be the angle between one of its asymptotes and the conjugate axis (Figure \ref{fig:KepHyp}). Standard formulas in celestial mechanics (see \cite{bate2020fundamentals}, \cite{celletti2010stability}) give 
\[
\sin\delta=\frac{1}{e},\quad e=\sqrt{1+2hC'^2}, 
\]
where $e$ is the eccentricity.  

\begin{figure}
        \centering        
        \begin{overpic}[height=0.2\textheight]{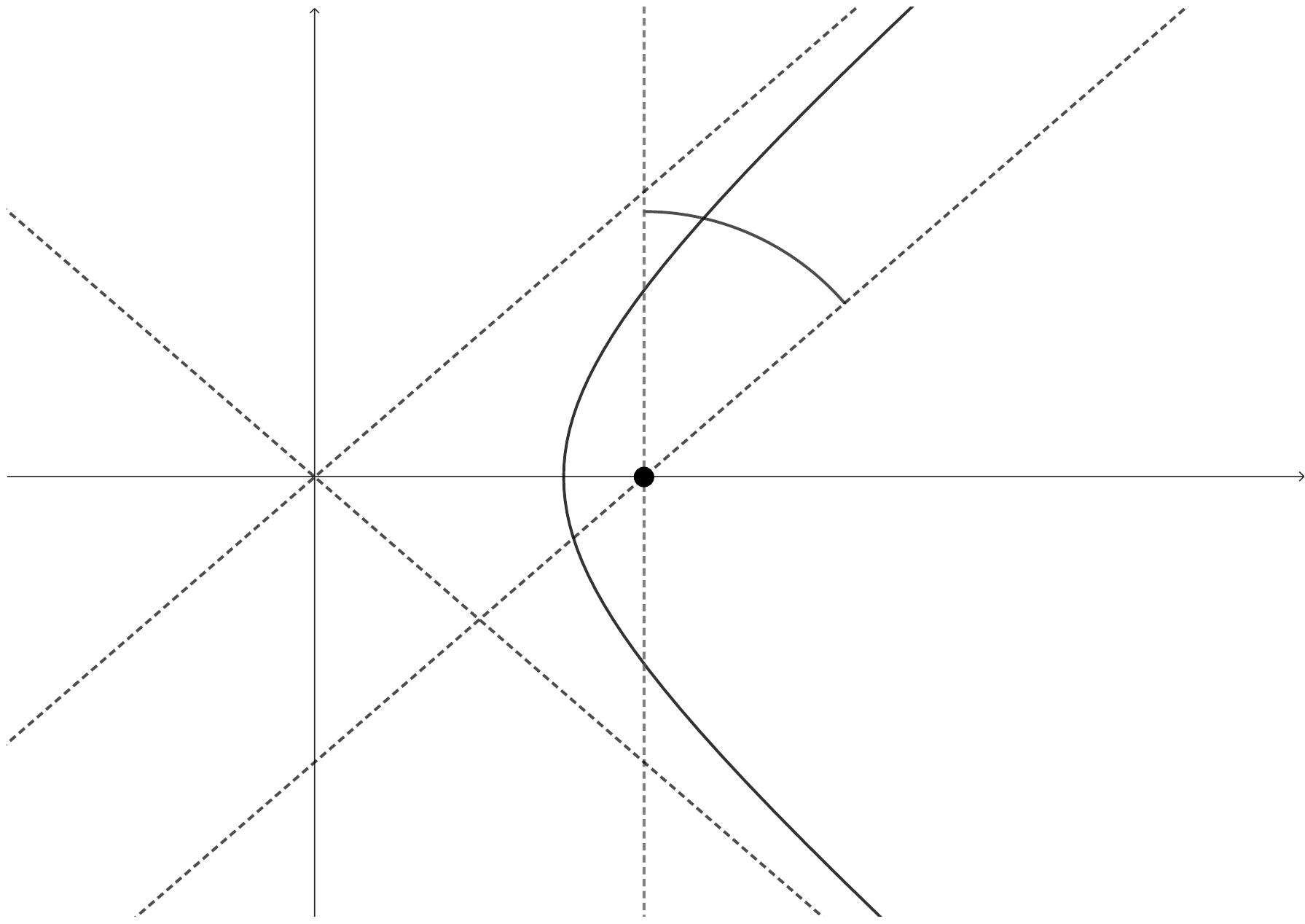}
		\put (55, 45) {\rotatebox{0}{$\delta$}}
        \end{overpic}
\caption{Keplerian hyperbola. The angle $\delta$ is defined between one asymptote and the conjugate axis.}
        \label{fig:KepHyp}
\end{figure}

For every $C\in K$, one has $\hat \theta(h, 0, C')\in (\pi/2, \pi)$. Thus, there exists $\epsilon\in(1/2,1)$ such that, for all $C\in K$ and $\beta<\beta_1$, 
\[
    \frac{\pi}{2}<\hat\theta(h, 0, C')<\epsilon \pi. 
\]
Moreover, for any $\epsilon'\in(0,\epsilon)$, there exists $\beta_2\in(0, \beta_1)$ such that, for $\beta\in(0,\beta_2)$ and all $C\in K$,
\[
    1<\frac{C}{C'}<\frac{1}{\epsilon'}.
\]
Hence, if $\beta\in(0,\beta_2)$, every orbit tangent to $\ell$ satisfies
\[
    \hat\theta(h, \beta, C)=\frac{C}{C'}\hat\theta(h, 0, C')<\frac{\epsilon}{\epsilon'}\pi<\pi. 
\]

This shows that such orbits cannot intersect $\ell$ again. Finally, let $P$ and $P'$ be the pericentres of $\mathcal O$ (orbit of $V_\beta$) and $\mathcal O'$ (Keplerian orbit with same tangency point), respectively. Clearly, $P'\in \{-L\leq y\leq 0\}$. By continuity of ODE solutions, $\|P-P'\|\to 0$ as $\beta\to 0$. Thus we can choose $\bar\beta \in (0,\beta_2)$ so that for all $\beta<\bar\beta$, $P$ also lies in $\{-L\leq y\leq 0\}$. Since $z_0\in \ell'$ (a compact segment), $\bar\beta$ can be chosen independent of $z_0$. 
\end{proof}

We conclude that, for $\beta$ sufficiently small, minimisers in $\mathcal H_{z_0z_1}^{k}$ do not touch $\partial D$ except at the endpoints. By compactness of $\ell'$, the threshold $\bar\beta$ is uniform in $z_0,z_1$ and in $k\neq 0$.  

Applying Theorem \ref{thm:M_crit_point}, we obtain the following existence result.

\begin{theorem}\label{thm:class_sol_k}
    There exists $\bar\beta=\bar\beta(L, h, \alpha)>0$ such that for every $\beta<\bar\beta$, $z_0, z_1\in \ell'$, and $k\in\mathbb Z$, there exists a classical solution of \eqref{eq:diffProb} with $z(0)=z_0$, $z(T)=z_1$ for some $T>0$, such that $z(t)\in \mathring{D}$ for all $t\in(0, T)$ and $\operatorname{Ind}(z)=k$. 
\end{theorem}

\section{Construction of Symbolic Dynamics for Boltzmann's Billiard}\label{sec:symDyn}

Trajectories from Theorem \ref{thm:class_sol_k} are the building blocks of our symbolic dynamics, where the symbols correspond to winding numbers. In particular, billiard trajectories will be constructed by concatenating solutions of the mechanical system associated with $V_\beta$ as given in Theorem \ref{thm:class_sol_k}. To decide whether two arcs are reflected into one another, we use the following variational characterization of reflection.  

\begin{lemma}\label{lem:refl_var}
    Let $z_0, z, z_2$ be three points on $\ell'$, connected by a concatenation of two arcs $\mathcal Z_1(\cdot; z_0, z)$ and $\mathcal Z_2(\cdot; z, z_2)$ as in Theorem \ref{thm:class_sol_k}, with $\mathcal Z_1(0; z_0, z)=z_0$, $\mathcal Z_1(T_1; z_0, z)=\mathcal Z_2(0; z, z_2)=z$, and $\mathcal Z_2(T_2; z, z_2)=z_2$ for some $T_1,T_2>0$. Then the concatenation satisfies the reflection law at $z$ if and only if 
\begin{equation}\label{eq:critPoint}
    \partial_z\Big(\mathcal L\big(\mathcal Z_1(\cdot; z_0, z)\big)+\mathcal L\big(\mathcal Z_2(\cdot; z, z_2)\big)\Big)\Big|_{z=z_1}=0. 
\end{equation}
\end{lemma}

\begin{proof}
The concatenation of $\mathcal Z_1$ and $\mathcal Z_2$ satisfies the reflection law at $z_1$ precisely when $\dot{\mathcal Z}_1(T_1)$ and $\dot{\mathcal Z}_2(0)$ have equal components in the direction of $\ell'$. This follows from \eqref{eq:critPoint} and the formulas (see \cite[Appendix A]{barutello2023chaotic})
\begin{equation}
    \begin{aligned}
    \partial_z\mathcal L\!\left(\mathcal Z_1(\cdot; z_0, z)\right)&=\sqrt{V_\beta(z)}\frac{\dot{\mathcal Z}_1(T_1; z_0, z)}{|\dot{\mathcal Z}_1(T_1; z_0, z)|}\cdot(1,0),\\   
    \partial_z\mathcal L\!\left(\mathcal Z_2(\cdot; z, z_2)\right)&=-\sqrt{V_\beta(z)}\frac{\dot{\mathcal Z}_2(0; z, z_2)}{|\dot{\mathcal Z}_2(0; z, z_2)|}\cdot(1,0). 
    \end{aligned}
\end{equation}
\end{proof}

We now construct the symbolic dynamics. Let $X$ denote the set of initial conditions $(z_0, v_0)\in \ell'\times \R^2$ of billiard trajectories associated with $V_\beta$ at energy $h>0$, lying in $D$ and starting from $\ell'$ without collisions.\footnote{At this stage, the set $X$ is defined by an \emph{a posteriori} reasoning on the trajectories generated from it. In the following we show that this definition is well-posed and $X\neq\emptyset$.}  

The first step is to prove the existence of non-collisional periodic trajectories inside $D$, with prescribed winding numbers around $0$.  

Let $\Omega$ be the set of bi-infinite sequences with entries in $\Z\setminus \{0\}$. Following \cite{hasselblatt2003first}, define on $\Omega$ the metric
\[
d(s, \bar{s}):= \sum_{i= -\infty}^{\infty} \frac{\delta(s_i, \bar{s}_i)}{\lambda^{|i|}}, \quad \lambda>3,
\]
where $\delta(i,j)=0$ if $i=j$ and $\delta(i,j)=1$ otherwise.  

The \emph{Bernoulli shift} $\sigma:\Omega \to \Omega$ is defined by 
\[
\sigma((s_i)_i)=(s_{i+1})_i.
\]
This transformation is a homeomorphism; its periodic points are dense in $\Omega$, and $\sigma$ is topologically mixing, hence chaotic (see again \cite{hasselblatt2003first}).  

We next establish a correspondence between sequences in $\Omega$ and billiard trajectories of our system, via a \emph{projection map} $\pi$.  

\begin{defi}\label{def:defPi}
    Let $(z_0, v_0)\in X$, and let $F$ be the first return map on $X$, sending $(z_0, v_0)$ to $(z_1, v_1)$, where the particle next hits $\ell$ and reflects (see Figure \ref{fig:defF}). Suppose the billiard trajectory starting from $(z_0, v_0)$ yields a bi-infinite sequence of reflection points $\{z_i\}_{i\in\Z}\subset\ell'$ with associated winding numbers $\{s_i\}_{i\in\Z}\subset\Z\setminus\{0\}$. We say the orbit starting from $(z_0, v_0)$ \emph{realises} the sequence $(\dots,s_{-1}, s_0, s_1, \dots )\in \Omega$.  

    In this case, the projection map $\pi$ is well defined by
    \[
    \pi(z_0, v_0)=(\dots,s_{-1}, s_0, s_1, \dots ).
    \]
\end{defi}

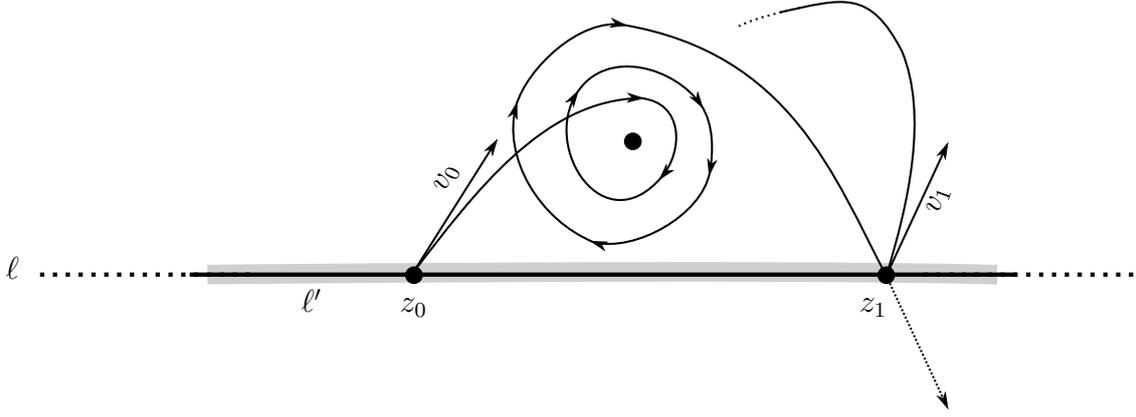
\begin{figure}
        \centering        
        \begin{overpic}[width=0.9\textwidth]{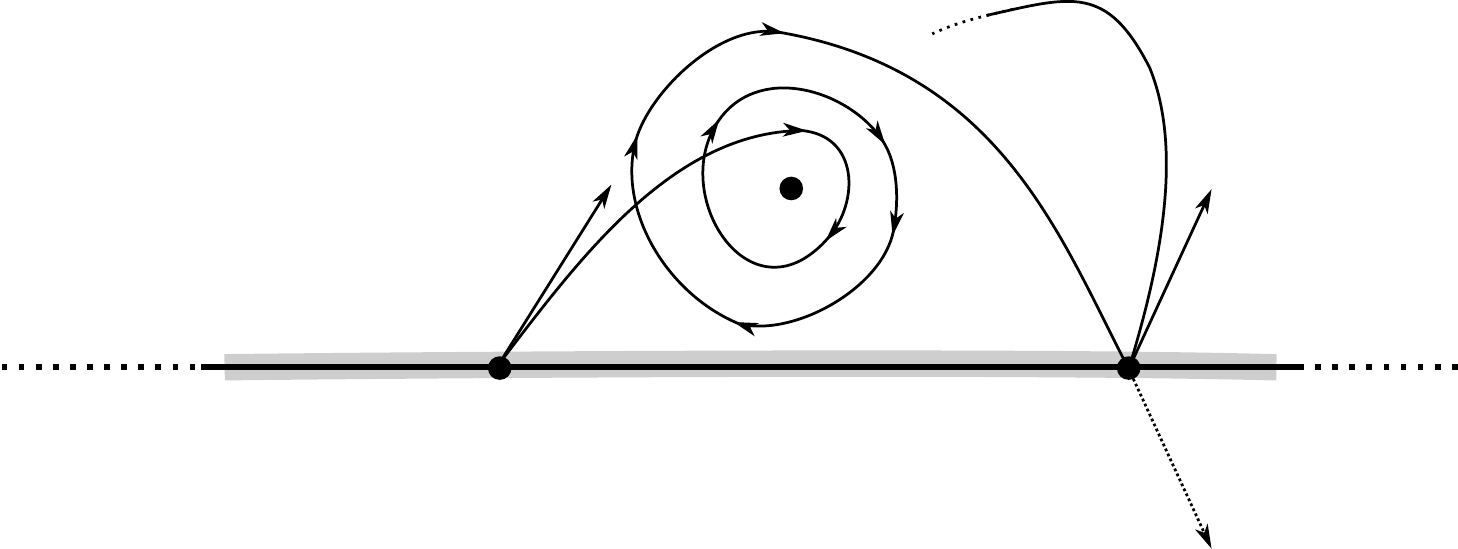}
        \put (-3, 12) {\rotatebox{0}{$\ell$}}
        \put (24,9) {\rotatebox{0}{$\ell'$}}
        \put (33,9) {\rotatebox{0}{$z_0$}}
        \put (36,20) {\rotatebox{65}{$v_0$}}
        \put (75,9) {\rotatebox{0}{$z_1$}}
        \put (81,18) {\rotatebox{70}{$v_1$}}
        \end{overpic}
\caption{First return map $F: (z_0, v_0)\mapsto (z_1, v_1)$ described in Definition \ref{def:defPi}. }
        \label{fig:defF}
\end{figure}

Clearly, $\pi$ is not well defined for every $(z_0,v_0)\in X$, but only for those initial conditions generating billiard trajectories realising sequences in $\Omega$. We therefore seek a subset $X'\subset X$ on which $\pi$ is well defined and yields a symbolic dynamics. We begin by proving that $\pi$ is surjective onto the set of periodic sequences.  

\begin{prop}   \label{prop: periodic_realization}
    Let $s \in \Omega$ be periodic of period $[s]=(s_0, s_1, \dots, s_{n-1})$. If $\beta<\bar\beta$, with $\bar\beta$ as in Theorem \ref{thm:class_sol_k}, then there exists a periodic billiard trajectory at energy $h$ consisting of $n$ arcs, where the $k$-th arc has winding number $s_k$ for $k=0,\dots,n-1$.
\end{prop}

\begin{proof}
By Theorem \ref{thm:class_sol_k}, for any $[z]=(z_0, z_1,\dots, z_{n-1}) \in \ell'^n$, there exists a concatenation $\mathcal{Z}^{[s]}_{[z]}$ of arcs $\mathcal{Z}^{[s]}_{z_k z_{k+1}}$, $k=0,\dots,n-1$, with $z_n=z_0$. Each arc is a collision-free minimiser of $\mathcal{M}$ in $H^{s_k}_{z_k z_{k+1}}$, lies in the interior of $D$, and has winding number $s_k$. Hence it is a classical solution of \eqref{eq:fixed_ends_prob} with endpoints $z_k,z_{k+1}$.  

We seek $[z]\in \ell'^n$ such that the concatenation $\mathcal{Z}^{[s]}_{[z]}$ forms a billiard trajectory. Define the total Jacobi length
\[
W_n^{[s]}([z]) = \sum_{k=0}^{n-1} \mathcal{L}(\mathcal{Z}^{[s]}_{z_k z_{k+1}}).
\]
Although the solutions $\mathcal{Z}^{[s]}_{z_kz_{k+1}}$ are not unique, their Jacobi length is uniquely determined, so $W_n^{[s]}$ is well defined.  

By Lemma \ref{lem:refl_var}, the concatenation is a billiard trajectory if and only if $[z]$ is a critical point of $W_n^{[s]}$. Since $\ell'^n$ is compact, $W_n^{[s]}$ admits a minimiser $[\tilde z]$. To conclude, it suffices to exclude minimisers lying on the boundary of $\ell'^n$, which follows from the next lemma. 
\end{proof}

\begin{lemma}
Let $\partial_{i}W^{[s]}_n$ denote the derivative of $W^{[s]}_n$ with respect to the $i$-th coordinate $(z_i)_x$, $i=1,\dots,n$. Denote the endpoints of $\ell'$ by $\hat z_0=(\hat x_0, -L)$ and $\hat z_1=(\hat x_1, -L)$ (see Figure \ref{fig:hat_z}). Then, for each $i=1,\dots,n$,
\begin{align}
         \partial_{i}W^{[s]}_n|_{z_i= \hat z_0} &<0, \label{eq: W_ineq1}\\
         \partial_{i}W^{[s]}_n|_{z_i= \hat z_1} &>0. \label{eq: W_ineq2}
\end{align}
\end{lemma}

\begin{proof}
By Maupertuis' principle, trajectories of the system $(\R^2,V_\beta,h)$ are reparametrisations of geodesics for the Jacobi metric $g_{ij}=(h-V_\beta)\delta_{ij}$. Consider an arc $\mathcal Z$ in $D$ with endpoints $z_0=\hat z_0$ and $z_1\in\ell'$, given by a minimiser $u\in H^k_{z_0z_1}$. Take a small ball $B$ centred at $u(0)$. Since $V_\beta$ is nearly constant in $B$, the Jacobi metric is approximated by a rescaled Euclidean metric in $B$.  

Let $\tilde z$ be the intersection of $u$ with $\partial B$ (see Figure \ref{fig:lem_change_sign}). By additivity of Jacobi length, the segment of $u$ between $z_0$ and $\tilde z$ is minimal. For sufficiently small $B$, the Euclidean distance satisfies $\|z_0-\tilde z\|>\|z-\tilde z\|$ for any $z\in \ell'\cap B$. By Lemma \ref{lem: boundary_existence}, the angle between $\dot{\hat{\mathcal Z}}(t_0)$ and $(1,0)$ is $<\pi/2$, ensuring $u$ remains in $D$. Therefore,
\[
\|u(0)-\tilde z\|>\|z-z_0\|.
\]
This comparison with Euclidean distance implies \eqref{eq: W_ineq1}. A symmetric argument yields \eqref{eq: W_ineq2}.
\end{proof}

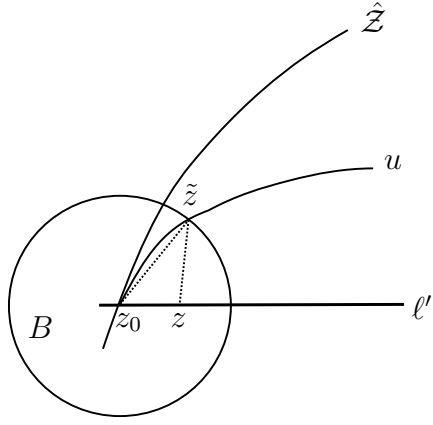
\begin{figure}
        \centering        
        \begin{overpic}[height=0.2\textheight]{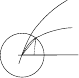}
		\put (89, 98) {\rotatebox{0}{$\hat{\mathcal Z}$}}
        \put (95, 63) {\rotatebox{0}{$u$}}
        \put (5, 20) {\rotatebox{0}{$B$}}
        \put (27, 23) {\rotatebox{0}{$z_0$}}
        \put (44, 54) {\rotatebox{0}{$\tilde z$}}
        \put (41,23) {\rotatebox{0}{$z$}}
        \put (102,25) {\rotatebox{0}{$\ell'$}}
        \end{overpic}
\caption{Local approximation of Jacobi length by Euclidean distance.} 
        \label{fig:lem_change_sign}
\end{figure}

Define $X'= \bigcap_{i= -\infty }^{\infty}F^{i}(X)$, where $\pi$ is well defined. By Proposition \ref{prop: periodic_realization}, $X'$ is non-empty (it contains all periodic billiard trajectories in $D$); moreover, $X'$ is $F$-invariant, and by construction $\pi \circ F = \sigma\circ \pi$.  

To prove that $\pi$ yields a symbolic dynamics between $F_{|_{X'}}$ and $\sigma$, we must show that $\pi$ is surjective and continuous.  

\begin{prop}
The map $\pi: X'\to \Omega$ is surjective. 
\end{prop}

\begin{proof}
Let $s\in \Omega$. We need a billiard trajectory realising $s$. Proposition \ref{prop: periodic_realization} covers the periodic case. Assume $s$ is not periodic. For each $n\in \N$, define the periodic sequence $s^{(n)}$ of period
\[
[s^{(n)}]= (s_{-n},\dots,s_0,\dots,s_{n}).
\]
Let $\mathcal Z^{(n)}$ be a billiard trajectory realising $s^{(n)}$. Denote by $x^{(n)}=\{x^{(n)}_j\}_{j\in\Z}$ the $x$-coordinates of the reflection points of $\mathcal Z^{(n)}$ on $\ell'$. For fixed $j$, the sequence $\{x^{(n)}_j\}_n$ is bounded, hence admits a convergent subsequence. By a diagonal extraction (see \cite{barutello2023chaotic}), there exists a subsequence $n_m$ and a sequence $\{\bar x_j\}_{j\in\Z}$ such that $x^{(n_m)}_j\to \bar x_j$ for each $j$. Setting $\bar z_j=(\bar x_j,-L)\in \ell'$, the differentiable dependence of solutions on initial data ensures that the reflection law still holds at every bounce. Thus there exists a billiard trajectory in $D$ consisting of arcs in $H^{s_k}_{\bar z_k\bar z_{k+1}}$, realising $s$.
\end{proof}

\begin{prop}\label{prop: time_uniform_bounded_allK}
Let $\mathcal Z$ be an arc of a billiard trajectory in $D$ with winding number $k\in\Z\setminus\{0\}$ and endpoints $\mathcal Z(t_0)=(x_0,-L)$, $\mathcal Z(t_1)=(x_1,-L)\in\ell'$. Then the travel time $T=|t_1-t_0|$ is uniformly bounded in $x_0,x_1$, and in $k$. 
\end{prop}

\begin{proof}
Let $r_{min}=\min_{t\in[t_0,t_1]}|\mathcal Z(t)|$ and $r_1=|\mathcal Z(t_1)|$. From \eqref{eq:PolEq}, the travel time from $r_{min}$ to $r_1$ is
\[
T=\frac{1}{\sqrt{2}}\int_{r_{min}}^{r_1} \frac{dr}{\sqrt{h-V_{eff}(r)}}=\frac{1}{\sqrt{2}}\int_{r_{min}}^{r_1}\frac{r\,dr}{\sqrt{hr^2+ r-A}},
\]
with $A=(C^2-2\beta)/2>0$ and $C$ the angular momentum.   The integral can be computed explicitly, recalling that $r_{min}$ is a root of the denominator: 
    \begin{equation*}
        \begin{aligned}
            T&=\frac{1}{2\sqrt{2}h}\left[\int_{r_{min}}^{r_1}\frac{2hr+1}{\sqrt{hr^2+1-A}}dr-\int_{r_{min}}^{r_1}\frac{dr}{\sqrt{hr^2+ r -A}}\right]\\
            &=\frac{\sqrt{h r_1^2+ r_1 -A}}{\sqrt{2}h}-\frac{1}{2\sqrt{2}h}\int_{r_{min}}^{r_1}\frac{dr}{\sqrt{\left(\sqrt{h}r+\dfrac{1}{2\sqrt{h}}\right)^2-\left(A+\dfrac{1}{4h}\right)}}. 
        \end{aligned}
    \end{equation*}
With the change of variables $x=\sqrt{h}r+1/(2\sqrt{h})$, and setting $B=A+1/(4h)$, $x_{min}=\sqrt{B}$ and $x_1=\sqrt{h}r_1+1/(2\sqrt{h})$, we find

    \begin{equation*}
        \begin{aligned}
        &\int_{r_{min}}^{r_1}\frac{dr}{\sqrt{\left(\sqrt{h}r+\dfrac{\mu}{2\sqrt{h}}\right)^2-\left(A+\dfrac{\mu^2}{4h}\right)}}= \dfrac{1}{\sqrt{h}}\int_{x_{min}}^{x_1}\frac{dx}{\sqrt{x^2-B}}\\
        &= \dfrac{1}{\sqrt{h}}\left[arccosh\left(\dfrac{x}{\sqrt{B}}\right)\right]_{x_{min}}^{x_1}=\dfrac{1}{\sqrt{h}}arccosh\left(\dfrac{2 h r_1 + \mu}{2 h (C^2-2\beta)+ \mu^2}\right). 
        \end{aligned}
    \end{equation*}

Since $C$ and $r_1$ range over compact sets, the bound on $T$ follows.
\end{proof}

\begin{prop}\label{prop: pi_continuity}
The map $\pi$ is continuous on $X'$.
\end{prop}

\begin{proof}
Fix $\varepsilon>0$. By the definition of the metric on $\Omega$,  for all $s,\bar s\in \Omega$, there exists $I\in \N$ such that,
\[
\sum_{|i|>I}\frac{\delta(s_i,\bar s_i)}{\lambda^{|i|}}<\varepsilon.
\]
Hence it suffices to ensure that, for any two billiard trajectories $\mathcal Z,\bar{\mathcal Z}$ starting from $\ell'$ with sufficiently close initial conditions, the winding numbers of the $i$-th arcs coincide for $|i|\leq I$. This follows from continuous dependence of solutions on initial conditions. Proposition \ref{prop: time_uniform_bounded_allK} guarantees that reflection times remain uniformly bounded, so the claim follows. Thus, $\pi$ is a continuous surjective map from $X'$ onto $\Omega$, intertwining the billiard map $F_{|_{X'}}$ with the Bernoulli shift $\sigma$. 
\end{proof}

The good definition of $\pi$ on a suitable subset of the initial conditions, along with its properties, allow us to state the final result our work, resumed in Theorem \ref{thm:final}, as well as Corollary \ref{cor:final}. 

\section{Conclusions and final remarks}\label{sec:conclusions}

The results presented in Theorem \ref{thm:final} and Corollary \ref{cor:final} demonstrate that introducing a strong-force correction to the Boltzmann billiard in the purely Keplerian case breaks its integrability, giving rise to a chaotic invariant set in the phase space of the billiard map.

In our framework, the possibility that the particle may, in principle, approach arbitrarily close to the centre of attraction is not a drawback. On the contrary, such close encounters generate winding trajectories, which serve as the foundation for our symbolic dynamics. Naturally, collisions remain a delicate technical issue, requiring precise estimates to connect the Jacobi length of minimising arcs with their pericentres.

It is worth noting that, although Boltzmann's original formulation does not specify the particle's position relative to the centre, most subsequent authors have preferred to place the particle in the opposite half-plane to avoid singularities. Nevertheless, as observed by Zhao \cite{Zhao2021}, the integrability of the Keplerian model is preserved on both sides.

A different perspective arises when considering the possible application of KAM theory, as discussed by Felder \cite{Felder}. There, the integrable Boltzmann billiard is shown to satisfy the Poncelet property, with level sets of the phase space characterised by families of periodic trajectories. Moreover, Felder constructs action-angle coordinates, enabling the application of perturbative KAM arguments, and thereby rules out ergodicity for sufficiently small values of $|\beta|$. In the geometry he considers, where the particle is uniformly bounded away from the centre by the presence of the wall, this reasoning is valid, since for small $\beta$ the potential $V_\beta$ indeed represents a perturbation of the integrable case. However, as already emphasised in the Introduction, such argument does not apply in our setting, where the particle may approach the centre arbitrarily closely.
Nevertheless, alternative considerations allow us to make progress regarding the system's ergodic properties. In particular, adopting the notion of quasi-ergodicity from \cite{Ehrenfest1911}, we observe that the invariant set $X'$ identified by our symbolic dynamics is strictly contained within the domain $X$ of the billiard map, since the trajectories under consideration lie within the bounded region $D$. Because of the arbitrariness in choosing the upper barrier $\hat{\mathcal Z}$, it is always possible to identify an open set in $X$ that is not contained in $X'$, and thus remains inaccessible from it. The situation is more subtle when adopting Birkhoff's definition of ergodicity \cite{Birkhoff1931}, which requires every invariant set to have either zero or full measure. At present, we cannot exclude the possibility that the invariant set we have constructed possesses zero measure; a deeper analysis of its structure will be required to resolve this question.

In this work, we have restricted attention to the case $\alpha, \beta, h > 0$, though the Boltzmann model can be studied well beyond these assumptions. In particular, our choice of positive energy implies that the dynamics induced by $V$ may be unbounded, leading to an ill-defined map for initial conditions corresponding to arcs that escape to infinity without returning to the reflecting table. This difficulty does not arise when $h<0$, since the Hill's region confines the particle's motion to a bounded subset of the plane. By contrast, for $h=0$, unbounded motions do occur, a case we plan to investigate in future work. Finally, we remark that our techniques do not extend to the case $\beta < 0$, where the presence of a centrifugal force requires a different approach, potentially based on the search for min-max rather than minimising solutions. \\
While, in principle, all the parameters of our problem, namely, $\alpha, \beta, L, h$ can change independently, the normalization leading to $\tilde V$ underlines how the real parameters of the problem are actually 
\[
\tilde \alpha= \frac{\alpha}{Lh}, \quad \tilde \beta=\tilde\alpha\frac{\beta}{L\alpha}. 
\]
This means that, keeping $\tilde \alpha$ constant (with obvious consequences on the relations between the original $\alpha$, $L$ and $h$), the small parameter of the problem is represented by $\beta(L\alpha)^{-1}$.  

\appendix

\section{Relationship between minimisers of $\mathcal M$ and $\mathcal L$}\label{sec:appA}

This Appendix is intended as a complement to prove that, given suitable hypotheses, a minimiser of the Maupertuis functional $\mathcal M$ minimises also the Jacobi length $\mathcal L$ in a general setting. Although more general in its treatise, such result is crucial in Section \ref{sec:fix_ends}, Corollary \ref{cor:minMminL}, to prove that minimisers of $\mathcal M$ in $\mathcal H_{z_0z_1}^k$ do not touch $\partial D$ except for their endpoints. The basic ideas resumed here come from the general theory of calculus of variations, and in particular from \cite{Cellina}. \\
Let us then start by taking $\Omega\subset\R^n$,   $p_0, p_1\in \Omega$, and 
\[
A= \left\{u\in H^{0,1}([0,1], \Omega) \ |\ u(0)=p_0 \text{ and }u(1)=p_1\right\}. 
\]
Then consider a potential $V:\Omega\to \mathbb R$ which is continuous and strictly positive.

\begin{lemma} \label{lem:NotZeroAE}
    Let $u\in A$; there exists a reparametrisation $\tilde u\in A$ such that $|\dot{\tilde u}(t)|\neq 0$ almost everywhere in $[0,1]$. In such case, $\mathcal M(\tilde u)\leq \mathcal M(u)$. As a consequence,  if $u$ is a minimiser of $\mathcal M$ in $A$, the set $\{t\in[0,1]\ |\ |\dot{u}(t)|=0\}$ has null measure in $[0,1]$. 
\end{lemma}

\begin{proof}
    Let us take $u\in A$, and consider the set $\mathcal I=\{t\in[0,1]\ |\ |\dot{u}(t)|=0\}$: let us observe that, whenever $\mathcal I$ contains an interval, it holds that $u(t)=const$ in such interval. The complementary $[0,1]\setminus\mathcal I$ is a countable union of open intervals, outside of which $u(t)$ is constant: let us then consider $u_{|_{[0,1]\setminus \mathcal I}}$, and reparametrise it into a new function $u_*: [0, b]\to \Omega$, $[0,b]\subseteq[0,1]$, by gluing together the pieces corresponding to each connected component. Since $u$ is constant in every connected component of $\mathcal I$, $u_*\in H^{0, 1}([0,b], \Omega)$, $u_*([0,b])=u([0,1])$ and it holds 
    \[
    \int_0^b |\dot u_*(t)|^2dt=\int_0^1|\dot u(t)|^2dt , \, \int_0^b V(u_*(t))dt\leq\int_0^1V(u(t))dt, 
    \]
    so that 
    \[
    \dfrac12   \int_0^b |\dot u_*(t)|^2dt \int_0^b V(u_*(t))dt\leq \mathcal M(u). 
    \]
    Let us now take a rescaling $\tilde u(t)=u_*(bt)$: it holds that $\tilde u\in A$, $\tilde u([0,1])=u([0,1])$ and 
    \[
    \mathcal M(\tilde u)=\dfrac12   \int_0^b |\dot u_*(t)|^2dt \int_0^b V(u_*(t))dt \leq \mathcal M(u). 
    \]
\end{proof}

\begin{lemma}\label{lem:L2equal2M}
    Let $u\in A$ such that $|\dot u(t)|\neq 0$ almost everywhere in $[0,1]$. Then there exists a reparametrisation $v$ of $u$ such that $\mathcal L(v)^2=2\mathcal M(v)$. 
\end{lemma}

\begin{proof}
    By Cauchy-Schwartz inequality, we know that for every $u\in A$ one has $\mathcal L^2(u)\leq 2\mathcal M(u)$, with equality if and only if there exists $k\in \mathbb R$ such that 
    \begin{equation}\label{eq:EqCond}
    |\dot u(t)|^2=k V(u(t)) \text{ for almost every } t\in[0,1]. 
    \end{equation}
    Let us then take $u\in A$ such that $|\dot u(t)|\neq 0$ almost everywhere in $[0,1]$, and take the new time parameter 
    \[
    s(t)=\int_0^t\frac{|\dot u(\xi)|}{\sqrt{2V(u(\xi))}}d\xi: 
    \]
    the parameter $s(t)$ is strictly increasing in $t$, hence its inverse $t(s)$ is well defined and $t'(s)=\sqrt{2V(u(t(s)))}/|\dot u(t(s))|$ almost everywhere in $[0, S]$, $S>0$ being a suitable constant.
    Let us consider the new reparametrisation $v(s)=u(t(s)): $ it holds 
    \[
    \dfrac12\Bigg|\dfrac{d}{ds}v(s)\Bigg|^2-V(v(s))=\dfrac12|\dot u(t)|^2 2 \frac{V(u(t))}{|\dot u(t)|^2}-V(v(s)) =0
    \]
    almost everywhere in $[0,S]$. Then condition \eqref{eq:EqCond} holds and $\mathcal L^2(v)=2\mathcal M(v)$. It is now sufficient to rescale $v$ in $[0,1]$ and the same equality holds.  
\end{proof}

\begin{prop}\label{prop:minML}
    Let $u\in A$ a minimiser of $\mathcal M$. Then $u$ minimises $\mathcal L$ as well. 
\end{prop}

\begin{proof}
    First of all, let us observe that, being $u$ a critical point for $\mathcal M$, it holds $\mathcal L^2(u)=2\mathcal M(u)$ (cfr. \cite{ambrosetti2012periodic}). \\
    Let us assume by contradiction that there exists $z\in A$ such that $\mathcal L(z) <\mathcal L(u)$. From Lemma \ref{lem:NotZeroAE}, possibly reducing the value of $\mathcal M$, we can assume that $|\dot z(t)|\neq 0$ almost everywhere in $[0,1]$, and then, by Lemma \ref{lem:L2equal2M}, we can find a reparametrisation $v$ of $z$ such that $\mathcal L^2(v)=2\mathcal M(v)$. As a consequence, 
    \[
    2 \mathcal M(z)=\mathcal L^2(z)<\mathcal L^2(u)=2\mathcal M(u),
    \]
    leading to a contradiction. 
\end{proof}

\begin{thebibliography}{10}

\bibitem{ambrosetti2012periodic}
Antonio Ambrosetti and Vittorio Coti-Zelati.
\newblock {\em Periodic solutions of singular Lagrangian systems}, volume~10.
\newblock Springer Science \& Business Media, 2012.

\bibitem{arnold1963kolmogorov}
V.~I. Arnold.
\newblock Proof of a theorem by Kolmogorov on the invariance of quasi-periodic motions under small perturbations of the Hamiltonian.
\newblock {\em Russian Mathematical Surveys}, 18:13--40, 1963.

\bibitem{BaranziniBarutelloDeBlasiTerracini2025}
Stefano Baranzini, Vivina~L. Barutello, Irene De~Blasi, and Susanna Terracini.
\newblock On the Birkhoff conjecture for Kepler billiards.
\newblock {\em arXiv preprint}, Jul 2025.

\bibitem{barutello2023chaotic}
Vivina~L Barutello, Irene De~Blasi, and Susanna Terracini.
\newblock Chaotic dynamics in refraction galactic billiards.
\newblock {\em Nonlinearity}, 36(8):4209, 2023.

\bibitem{bate2020fundamentals}
Roger~R Bate, Donald~D Mueller, Jerry~E White, and William~W Saylor.
\newblock {\em Fundamentals of astrodynamics}.
\newblock Courier Dover Publications, 2020.

\bibitem{Birkhoff1927}
G.~D. Birkhoff.
\newblock On the periodic motions of dynamical systems.
\newblock {\em Acta Mathematica}, 50(1):359--379, 1927.

\bibitem{Birkhoff1931}
G.~D. Birkhoff.
\newblock Proof of the ergodic theorem.
\newblock {\em Proceedings of the National Academy of Sciences of the United States of America}, 17:656--660, 1931.

\bibitem{Bolotin2017DegenerateBilliards}
S.~V. Bolotin.
\newblock Degenerate billiards in celestial mechanics.
\newblock {\em Regular and Chaotic Dynamics}, 22(1):27--53, 2017.

\bibitem{Boltzmann}
Ludwig Boltzmann.
\newblock L{\"o}sung eines mechanischen problems.
\newblock {\em Wiener Berichte}, 58:1035--1044, 1868.

\bibitem{boscaggin2023maupertuis}
Alberto Boscaggin, Walter Dambrosio, and Eduardo Mu{\~n}oz-Hern{\'a}ndez.
\newblock A Maupertuis-type principle in relativistic mechanics and applications.
\newblock {\em Calculus of Variations and Partial Differential Equations}, 62(3):95, 2023.

\bibitem{brezis2011functional}
Haim Brezis and Haim Br{\'e}zis.
\newblock {\em Functional analysis, Sobolev spaces and partial differential equations}, volume~2.
\newblock Springer, 2011.

\bibitem{brin2002introduction}
Michael Brin and Garrett Stuck.
\newblock {\em Introduction to dynamical systems}.
\newblock Cambridge university press, 2002.

\bibitem{celletti2010stability}
Alessandra Celletti.
\newblock {\em Stability and chaos in celestial mechanics}.
\newblock Springer Science \& Business Media, 2010.

\bibitem{Cellina}
Arrigo Cellina, Giulia Treu, and Sandro Zagatti.
\newblock On the minimum problem for a class of non-coercive functionals.
\newblock {\em J. Differ. Equations}, 127(1):225--262, 1996.

\bibitem{Dragovic1998}
Vladimir Dragovi\'c.
\newblock Classical integrable mechanical systems and their integrable perturbations.
\newblock {\em Publications de l'Institut Math\'ematique. Nouvelle S\'erie}, 64(78):153--164, 1998.

\bibitem{Ehrenfest1911}
Paul Ehrenfest and Tatjana Ehrenfest-Afanassjewa.
\newblock Begriffliche grundlagen der statistischen auffassung in der mechanik.
\newblock In {\em Encyklop{\"a}die der mathematischen Wissenschaften mit Einschluss ihrer Anwendungen}, volume~4, pages 4--90. Teubner, Leipzig, 1911.

\bibitem{fedorov2001ellipsoidal}
Yu~N Fedorov.
\newblock An ellipsoidal billiard with a quadratic potential.
\newblock {\em Functional Analysis and Its Applications}, 35(3):199--208, 2001.

\bibitem{Felder}
Giovanni Felder.
\newblock Poncelet property and quasi-periodicity of the integrable Boltzmann system.
\newblock {\em Letters in Mathematical Physics}, 111(1):12, 2021.

\bibitem{Gallavotti-Jauslin}
Giovanni Gallavotti and Ian Jauslin.
\newblock A theorem on ellipses, an integrable system and a theorem of Boltzmann.
\newblock {\em arXiv preprint arXiv:2008.01955}, 2020.

\bibitem{Gasiorek}
Sean Gasiorek and Milena Radnovi\'c.
\newblock {\em Periodic trajectories and topology of the integrable Boltzmann system}, pages 111--130.
\newblock 01 2024.

\bibitem{hasselblatt2003first}
Boris Hasselblatt and Anatole Katok.
\newblock {\em A first course in dynamics: with a panorama of recent developments}.
\newblock Cambridge University Press, 2003.

\bibitem{Henon1964}
Michel H\'enon and Carl Heiles.
\newblock The applicability of the third integral of motion: some numerical experiments.
\newblock {\em The Astronomical Journal}, 69:73--79, 1964.

\bibitem{JAUD2024105289}
Daniel Jaud and Lei Zhao.
\newblock Geometric properties of integrable Kepler and Hooke billiards with conic section boundaries.
\newblock {\em Journal of Geometry and Physics}, 204:105289, 2024.

\bibitem{jose1998classical}
Jorge~V Jos{\'e} and Eugene~J Saletan.
\newblock {\em Classical dynamics: a contemporary approach}.
\newblock Cambridge university press, 1998.

\bibitem{KaloshinSorrentino2018}
Vadim Kaloshin and Alfonso Sorrentino.
\newblock On the integrability of birkhoff billiards.
\newblock {\em Philosophical Transactions of the Royal Society A: Mathematical, Physical and Engineering Sciences}, 376(2131):20170419, 2018.

\bibitem{kolmogorov1954hamiltonian}
A.~N. Kolmogorov.
\newblock On the conservation of conditionally periodic motions under small perturbation of the Hamiltonian.
\newblock {\em Doklady Akademii Nauk SSSR}, 98:527--530, 1954.

\bibitem{marino1983geodesics}
Antonio Marino and D~Scolozzi.
\newblock Geodesics with obstacle.
\newblock {\em Bollettino dell'Unione Matematica Italiana. B}, 6:1--31, 1983.

\bibitem{moser1962invariant}
J.~K. Moser.
\newblock On invariant curves of area-preserving mappings of an annulus.
\newblock {\em Nachrichten der Akademie der Wissenschaften in G\"ottingen, II. Mathematisch-Physikalische Klasse}, 1:1--20, 1962.

\bibitem{Panov1994EllipticalBilliard}
A.~A. Panov.
\newblock Elliptical billiard table with Newtonian potential.
\newblock {\em Mathematical Notes}, 55(3):334, 1994.

\bibitem{poincare1891trois}
Henri Poincar\'e.
\newblock Le probl\`eme des trois corps.
\newblock {\em Revue g\'en 'erale des sciences pures et appliqu\'ees}, 2:15, 1891.

\bibitem{Poritsky1950}
H.~Poritsky.
\newblock The billiard ball problem on a table with a convex boundary---an illustrative dynamical problem. I. The invariant integral.
\newblock {\em Annals of Mathematics}, 51:446--470, 1950.

\bibitem{Reichert2024}
Paula Reichert.
\newblock {\em The Ergodic Hypothesis: A Typicality Statement}, pages 285--299.
\newblock Springer International Publishing, Cham, 2024.

\bibitem{Takeuchi-Zhao2023}
Airi Takeuchi and Lei Zhao.
\newblock Projective integrable mechanical billiards.
\newblock {\em Nonlinearity}, 37(1):015011, 2023.

\bibitem{TAKEUCHI2024109411}
Airi Takeuchi and Lei Zhao.
\newblock Conformal transformations and integrable mechanical billiards.
\newblock {\em Advances in Mathematics}, 436:109411, 2024.

\bibitem{veselov1991integrable}
Aleksandr~Petrovich Veselov.
\newblock Integrable maps.
\newblock {\em Russian Mathematical Surveys}, 46(5):1, 1991.

\bibitem{Zhao2021}
Lei Zhao.
\newblock Projective dynamics and an integrable Boltzmann billiard model.
\newblock {\em Communications in Contemporary Mathematics}, 24(10):2150085, 2022.

\end{thebibliography}

\end{document}